\documentclass[11pt]{article}

\usepackage{amsmath,amssymb,amsthm,graphicx}
\usepackage{a4wide}
\usepackage{caption}
\usepackage{subcaption}
\usepackage{tikz}
\usepackage{authblk}

\newcommand{\lma}{\lambda_p}

\newcommand{\Sp}{S^{V}_p}
\newcommand{\pn}[1]{\left\lVert #1 \right \rVert_p}
\newcommand{\R}{\mathbb{R}}

\newcommand{\sg}[1]{\textnormal{sign}\left(#1\right)}

\newcommand{\cut}{\textnormal{cut}}

\newtheorem{lemma}{Lemma}
\newtheorem{prop}[lemma]{Proposition}
\newtheorem{coro}[lemma]{Corollary}
\newtheorem{theorem}[lemma]{Theorem}

\numberwithin{lemma}{section}
\numberwithin{fact}{section}
\numberwithin{equation}{section}

\title{Eigenvalue bounds for the signless $p$-Laplacian}

\author[1]{Elizandro Max Borba}
\author[2]{Uwe Schwerdtfeger}

\affil[1]{Instituto de Matem\'atica, Universidade Federal do Rio Grande do Sul\\

CEP 91509-900, Porto Alegre, RS, Brazil\\

\textit {elizandro.max@ufrgs.br}}

\affil[2]{Fakult\"at f\"ur Mathematik, Technische Universit\"at Chemnitz\\

D-09107 Chemnitz, Germany\\

\textit {uwe.schwerdtfeger@mathematik.tu-chemnitz.de}}

\begin{document}
\maketitle

\begin{abstract}
We consider the signless $p$-Laplacian of a graph, a generalisation of the usual signless Laplacian (the case $p=2$). We show a Perron-Frobenius property and basic inequalites for the largest eigenvalue and provide upper and lower bounds for the smallest eigenvalue in terms of a parameter related to the bipartiteness. The latter result generalises bounds by Desai and Rao and, interestingly, in the limit $p\to 1$ upper and lower bounds coincide.
 


\end{abstract}

\section{Introduction}

We begin with some notation. All graphs are simple and undirected without loops or multiple edges. For a graph $G=(V,E)$ with vertex set $V$ the edge set $E$ hence consists of two-element subsets of $V$ and we shall frequently write $ij$ rather than $\{i,j\}$ for an edge. For disjoint subsets $S,T\subseteq V$ define $E_G(S)$ as the edges $ij\in E$ spanned by $S,$ $E_G(S,T)$ the edges with one vertex in $S$ and the other in $T,$ furthermore we define $n=|V|,$ $e_G(S)=|E_G(S)|,$ $e_G(S,T)=|E_G(S,T)|$ and $\cut_G(S)=|E_G(S,V\setminus S)|.$ Furthermore $d_i,i\in V$ is the degree of vertex $i$ and $\delta(G),\Delta(G)$ denote the minimum and maximum degree, respectively.  We shall drop the subscripts if $G$ is clear from the context. Let $A$ be the adjacency matrix and $D$ be the diagonal matrix of vertex degrees. We also find it convenient to index vectors and matrices associated with a graph by the vertex set $V$ and write therefore $x=(x_i,\in V)\in \R^V$ (rather than $x\in \R^n$) and consider accordingly $A$ 
and $D$ as 
operators $\R^V\to \R^V,$ similarly for $x\in \R^V$ we understand $x^\top$ as the functional $\R^V\to\R,$ $y\mapsto x^\top y=\sum_{i\in V}x_iy_i.$ 

The eigenvalues and -vectors of the Laplacian $L:=D-A$ are well studied, in particular the second smallest eigenvalue $a(G),$ the algebraic connectivity of $G.$ It is non-zero if and only if $G$ is connected and by a well-known inequality due to Mohar \cite{Mohar89} it can be upper and lower bounded in terms of the isoperimetric number of $G$
\[
 i(G)=\min\left\{\frac{\cut(S)}{|S|},\;S\subseteq V,\;0<|S|\le\frac{n}{2} \right\}.
\]
Eigenvectors for $a(G)$ are used in spectral partitioning. Following Amghibech's work \cite{Amghibech03} B\"uhler and Hein \cite{BuehlerHein09} introduced for $p>1$ a non-negative functional $L_p(x)=\sum_{ij\in E}|x_i-x_j|^p$ on $\R^V$ (for $p=2$ the quadratic form of $L$) and called the non-linear operator 
\[
\frac{1}{p}\nabla_x L_p\colon \R^V\to \R^V,\quad \frac{1}{p}\nabla_x L_p(x)_i= \sum_{j\colon ij\in E}\sg{x_i-x_j} |x_i-x_j|^{p-1}
\]
the $p$-Laplacian. A pair $(\mu,x),$ $x\neq 0,$ is called an eigenvalue and -vector of $L_p$ if the eigenequations
\[
\frac{1}{p}\nabla_x L_p(x)_i=\sum_{j\colon ij\in E}\sg{x_i-x_j} |x_i-x_j|^{p-1}=\mu\sg{x_i}x_i^{p-1},\,i\in V,
\]
are satisfied where \[\sg{x}=\begin{cases}1 \text{ if }x>0,\\ -1\text{ if }x <0,\\ 0\text{ if }x =0.\end{cases}\] Observe that in this case $p^{-1}x^\top\nabla_x L_p(x)=L_p(x)=\mu \pn{x}^p$ and that the eigenequations are necessary conditions for the optimisation problems
\[
 \min\,(\text{resp. } \max) L_p(x) \text{ s.t. }\pn{x}^p=1
\]
which for $p=2$ are the Rayleigh-Ritz characterisations of the smallest (resp. largest) eigenvalue of $L.$ The minimum is always zero, attained by a non-zero multiple of the all ones vector ${\bf 1}$ and these are the only solutions if and only if $G$ is connected. Observing that ${\bf 1}^\top \nabla_x L_p(x)=0$ for any $x$ one is lead to characterise the second smallest eigenvalue $a_p(G)$ as
\[
 a_p(G)=\min L_p(x) \text{ s.t. }\pn{x}^p=1 \text{ and }\sum_{i\in V}\sg{x_i}|x_i|^{p-1}=0
\]
B\"uhler and Hein \cite{BuehlerHein09} give the following bounds 
\begin{equation}\label{eq:Cheeger}
 \left(\frac{2}{\Delta} \right)^{p-1}\left(\frac{i(G)}{p} \right)^{p}\le a_p(G) \le 2^{p-1}i(G),
\end{equation}
where $\Delta$ is the maximum degree.
Interestingly, in the limit $p\to 1$ upper and lower bounds coincide and they show that thresholding of the eigenvector yields in the limit the optimal cut.

Only recently the signless Laplacian matrix $Q=D+A$ has received greater interest from spectral graph theorists, although the first remarkable result dates back to 1994. The quadratic form of $Q$ is $x^\top Q x=\sum_{ij\in E}(x_i+x_j)^2$ and hence $Q$ is positive semidefinite. The smallest eigenvalue $q(G)$ is $0$ if and only if $G$ has a bipartite component. The results in \cite{DesRao94} state roughly that a small $q(G)$ indicates the existence of a (near) bipartite subgraph which is not very well connected to the rest of the graph, and vice versa. More precisely, define the parameter $\psi(G)$ as 
\begin{equation}\label{eq:psi}
 \psi(G)=\min\left\{\frac{2e(S)+2e(T)+\cut(S\cup T)}{|S\cup T|}\colon S,T\subseteq V,\;S\cap T=\emptyset,\;S\cup T\neq\emptyset    \right\}.
\end{equation}
Then we have upper and lower bounds reminiscent of \eqref{eq:Cheeger} for $p=2,$ namely
\begin{equation}\label{eq:DesaiRao}
 \frac{\psi(G)^2}{2\Delta}\le q(G)\le 2 \psi(G)
\end{equation}
Observe that for a minimising pair $S,T$ in \eqref{eq:psi} the number $e(s)+e(T)$ is the smallest number of edges to be removed from the induced subgraph on $S\cup T$ to make it bipartite and thus the bounds quite well capture the informal statement.
We also remark that the definition of $\psi$ and also the bounds differ (superficially) from those in \cite{DesRao94}. We followed the exposition of \cite{FallatFan12} where bounds on $q(G)$ in terms of edge and vertex bipartiteness and a stronger lower bound in terms of $\psi$ are established. Finding (near) bipartite substructures in graphs is of practical interest in the study of social networks and bioinformatics as pointed out in \cite{KirklandPaul11}, where the authors devise spectral techniques involving eigenvectors for $q(G)$ to obtain such structures.

In this note we consider for real $p\ge 1$ the non-negative convex functional 
\[
Q_p\colon \R^V\to \R,\quad Q_p(x)=\sum_{ij\in E}|x_i+x_j|^p
\] 
(i.e. $Q_2(x)=x^\top Qx$) and its extremal values on the $p$-norm unit sphere 
\[\Sp=\left\{x\in \R^V\colon \sum_{i\in V}|x_i|^p=1 \right\}
\] 
or, equivalently, extremal values of $R_p(x):=Q_p(x)/\pn{x}^p$ on $\R^V\setminus\{0\}.$ For $p=2$ this is the Rayleigh-Ritz characterisation of the smallest and largest eigenvalues.  Let 
\begin{equation}\label{eq:q}
q_p(G)= \min\left\{ Q_p(x) \colon \pn{x}^p=\sum_{i\in V}|x_i|^p=1 \right\}=\min\left\{ R_p(x) \colon x\neq 0 \right\}
\end{equation}
and in an analogous fashion define by $\lma(G)$ the maximum. We call $q_p(G)$ and $\lma(G)$ the smallest and largest eigenvalue of $Q_p,$ respectively and minimising (maximising) vectors eigenvectors. Our study of $Q_p$ is strongly motivated by the inequalities \eqref{eq:DesaiRao} and \eqref{eq:Cheeger} and our main result is
\begin{theorem}\label{theo:main} For $p\ge 1$ the smallest signless $p$-Laplacian eigenvalue $q_p(G)$ satisfies
\[
 \left(\frac{2}{\Delta} \right)^{p-1}\left(\frac{\psi(G)}{p} \right)^{p}\le q_p(G) \le 2^{p-1}\psi(G).
\]
In particular, we have that $q_1(G)=\psi(G).$ 
\end{theorem}
The proof of Theorem \ref{theo:main} also shows how a ''good`` pair $(S,T)$ (in view of \eqref{eq:psi}) can be obtained by thresholding a solution to the optimisation problem \eqref{eq:q}: 
\begin{coro}\label{coro:main}
 For $p\ge 1$ let $x=x(p)$ be an eigenvector for $q_p(G)$ (i.e. $Q_p(x)=q_p(G)\pn{x}^p$) and define for $t\ge 0$ the vertex sets $S_{x}^t=\{i\in V\colon x_i>t\}$ and $T_{x}^t=\{i\in V\colon x_i<-t\}$ and let furthermore 
\[
\psi(x)=\min\left\{ \frac{2e_G\left(S_{x}^t\right)+2e_G\left(T_{x}^t\right)+\textnormal{cut}_G\left(S_{x}^t\cup T_{x}^t\right)}{\left|S_{x}^t\cup T_{x}^t\right|}\colon 0\le t< \max_{i\in V}\{|x_i|\} \right\}.
\]
Then $\psi\left(x(p)\right)\to \psi(G)$ as $p\to 1$ and in particular for $p=1,$ $\psi(G)=\psi(x(1)).$  
\end{coro}

\noindent
{\bf Outline.} In the next section we provide a proof of the theorem and its corollary. In the subsequent section we discuss some basic properties of the largest eigenvalue. 

\section{Proof of Theorem \ref{theo:main}}

In this section we assume throughout that $p\ge 1$. First we shall see that Theorem \ref{theo:main} is trivially true for graphs with a bipartite connected component:
\begin{lemma}
 The smallest eigenvalue $q_p(G)$ is zero if and only if $G$ has a bipartite component.
\end{lemma}
\begin{proof}
 If there is a bipartite component with partition $S\cup T$ put $x_i=1,i\in S,$ $x_i=-1,i\in T$ and $x_i=0$ elsewhere. Then $Q_p(x)=0.$ Conversely, $Q_p(x)=0$ implies $x_i=-x_j$ for every $ij\in E$ and thus a connected component containing a vertex $i$ with $x_i\neq 0$ is bipartite.
\end{proof}
The next lemma (\cite{DesRao94,FallatFan12} for $p=2$) yields the upper bound in Theorem \ref{theo:main}.
\begin{lemma}\label{lem:upperbound}
Let $S,T\subseteq V,$ $S\cap T=\emptyset.$ Then we have
 \[
  q_p(G)|S\cup T|\le 2^p e(S)+ 2^p e(T) +\cut(S\cup T)
 \]
and in particular $q_p(G)\le 2^{p-1}\psi(G).$
\end{lemma}
\begin{proof}
Let $S\cup T\neq\emptyset,$ otherwise the assertion is trivial. Define $x$ by $x_i=1,i\in S,$ $x_i=-1,i\in T$ and $x_i=0$ elsewhere. Then a computation shows
\[
q_p(G)\le \frac{Q_p(x)}{\pn{x}^p}=\frac{2^p e(S)+ 2^p e(T) +\cut(S\cup T)}{|S\cup T|}.
\]
If $S,T$ are chosen as optimal sets in the definition \eqref{eq:psi} of $\psi$ this last expression is $\le 2^{p-1}\psi.$
\end{proof}
For the proof of the lower bound we proceed as in \cite{DesRao94} and \cite{BuehlerHein09}. For a graph $G^\prime=(V^\prime,E^\prime)$ fix a subset $U\subseteq V^\prime.$ 
For a vector $g\in \R^{V^\prime}$ with $g_i>0$ on $U$ and $g_i=0$ elsewhere define $C^t_g=\{ i\in U\colon g_i>t \}$ and
\begin{equation}\label{eq:defhg}
 h_g(U)=\min \left\{\frac{\cut_{G^\prime}\left(C^t_g\right)}{\left|C^t_g\right|}\colon 0\le t < \max\{x_i,i\in U\} \right\}.
\end{equation}
We have the following lemma.
\begin{lemma}\label{lem:est} With the above notation and $\Delta(G^\prime)$ the maximum degree of $G^\prime$ we have 
 \begin{equation}
  \left(\frac{2}{\Delta(G^\prime)} \right)^{p-1}\left(\frac{h_g(U)}{p} \right)^{p}\pn{g}^p\le\sum_{ij\in E^\prime}|g_i-g_j|^p.
 \end{equation}
\end{lemma}
We postpone the proof of the lemma and first apply it to prove Theorem \ref{theo:main}. The idea is from \cite{DesRao94}, however, at some places we must be a bit more careful to get a lower bound which, as $p\to 1,$ coincides with the upper bound. 
\begin{proof}[Proof of Theorem \ref{theo:main}, lower bound]
 Consider a connected graph $G=(V,E)$ and its signless $p$-Laplacian $Q_p.$ Let $x$ be a normalised eigenvector ($\pn{x}=1$) for the smallest eigenvalue $q_p=q_p(G)$ such that $Q_p(x)=q_p.$ Define $S=\{i\in V\colon x_i>0 \}$ and $T=\{i\in V\colon x_i<0 \}$ and a graph $G^\prime =(V^\prime,E^\prime)$ as follows. Let $V^\prime=V\dot\cup S^\prime \dot\cup T^\prime$ where $S^\prime=\{i^\prime\colon i\in S\}$ and $T^\prime=\{i^\prime\colon i\in T\}$ are disjoint copies of $S$ and $T,$ respectively, and define
\[
 E^\prime=E_G(S,T)\cup E_G(S\cup T, V\setminus (S\cup T))\cup\{i^\prime j, ij^\prime\colon ij\in E_G(S)\} \cup\{i^\prime j, ij^\prime\colon ij\in E_G(T)\}
\]
i.e. $E^\prime$ is obtained from $E$ by deleting every edge $ij$ with both endpoints in $S$ (resp. $T$) and adding two edges $ij^\prime$ and $i^\prime j.$
Define $g\in\R^{V^\prime}$ by $g_i=|x_i|$ if $i\in S\cup T$ and $g_i=0$ if $i\in V^\prime\setminus (S\cup T).$ Then we have $\pn{g}=1$ and
\begin{equation}\label{eq:gx}
\sum_{ij\in E^\prime}|g_i-g_j|^p\le \sum_{ij\in E}|x_i+x_j|^p.
\end{equation}
To see this, first consider an edge $ij\in E_G(S)\cup E_G(T)$ on the right hand side. We have two corresponding edges in $i^\prime j ,ij^\prime\in E^\prime$ on the left side and 
\[
|g_{i^\prime}-g_j|^p+|g_i-g_{j^\prime}|^p=|g_j|^p+|g_i|^p\le ||x_i|+|x_j||^p=|x_i+x_j|^p
\]
because $x_i$ and $x_j$ have the same sign. The remaining summands on both sides correspond to edges $ij\in E\cap E^\prime.$ A case by case inspection yields $|g_i-g_j|^p=|x_i+x_j|^p.$\footnote{Inequality \eqref{eq:gx} is sharper than the original \cite[inequality (23)]{DesRao94} where there is a factor 2 on the right hand side which can be omitted. This was also pointed out in \cite{FallatFan12}.}

Now we show $h_g(S\cup T)\ge\psi(G).$ To that end consider the optimal cut set $C^t_g\subseteq S\cup T$ in \eqref{eq:defhg} and define $S^t=C^t_g\cap S$ and $T^t=C^t_g\cap T.$ Then, in $G,$ we have
\[
 2e_G(S^t)+2e_G(T^t)+\cut_G(S^t\cup T^t)\ge\psi(G) |S^t\cup T^t|
\]
Now, in $G^\prime,$ let us determine $\cut_{G^\prime}(S^t\cup T^t):$ every edge of $ij\in E_G(S^t)\cup E_G(T^t)$ contributes two edges to $\cut_{G^\prime}(S^t\cup T^t),$ namely $ij^\prime$ and $i^\prime j.$ Every edge in $E_G(S^t\cup T^t,V\setminus (S^t\cup T^t))$ contributes exactly one edge: for example, $ij\in E$ with $i\in S^t$ and $j\in S\setminus S^t$ in $G$ is accounted for in $\cut_{G^\prime}(S^t\cup T^t)$ by the edge $ij^\prime\in E^\prime,$ while the edge $i^\prime j\in E^\prime(V^\prime\setminus (S^t\cup T^t))$ is not. So we have
\begin{equation}\label{eq:hgpsi}
h_g(S\cup T)= \frac{\cut_{G^\prime}(S^t\cup T^t)}{|S^t\cup T^t|}=\frac{2e_G(S^t)+2e_G(T^t)+\cut_G(S^t\cup T^t)}{|S^t\cup T^t|}\ge \psi(G). 
\end{equation}
Now combine Lemma \ref{lem:est} (applied to $G\prime$ with $U=S\cup T$) with \eqref{eq:gx} and \eqref{eq:hgpsi} and observe that $\Delta(G^\prime)=\Delta(G)$ to complete the proof of Theorem \ref{theo:main}. Finally observe that $S^t,T^t$ are precisely the sets $S_x^t,T_x^t$ in Corollary \ref{coro:main} and we have actually proven a sharper lower bound with $\psi(G)$ replaced by $\psi(x).$ 
\end{proof}

We now prove the lemma, where we follow \cite{BuehlerHein09} up to minor modifications.
\begin{proof}[Proof of Lemma \ref{lem:est}] We first show
\begin{equation}\label{eq:sumgeh}
h_g(U)\pn{g}^p\le \sum_{ij\in E^\prime}|g_i^p-g_j^p| .
\end{equation}
To that end write
\[
\begin{split}
 \sum_{ij\in E^\prime,g_i>g_j}(g_i^p-g_j^p)=\sum_{ij\in E^\prime,g_i>g_j}p\int_{g_j}^{g_i}t^{p-1}dt=p\int_0^\infty t^{p-1}\sum_{ij\colon g_i>t\ge g_j}1dt
\end{split}
 \]
and observe that $\sum_{ij\colon g_i>t\ge g_j}1=|\{ij\in E^\prime\colon i\in C^t_g, j\notin C^t_g\}|=\cut(C^t_g).$ By the definition of $h_g(U)$ we have 
\[
\cut(C^t_g)\ge h_g(U) |C^t_g|=h_g(U)\sum_{i\in V^\prime\colon g_i>t}1.
\]
Hence we can estimate
\[\begin{split}
 \sum_{ij\in E^\prime,g_i>g_j}(g_i^p-g_j^p)&=p\int_0^\infty t^{p-1}\sum_{ij\colon g_i>t\ge g_j}1dt\\ &\ge p\int_0^\infty t^{p-1}h_g(U)\sum_{i\in V^\prime\colon g_i>t}1 dt\\ &=h_g(U)\sum_{i\in V^\prime\colon g_i>0}\int_0^{g_i}pt^{p-1}dt\\ &=h_g(U)\sum_{i\in V^\prime\colon g_i>0}g_i^p\\ &=h_g(U)\pn{g}^p
\end{split}\]
and obtain \eqref{eq:sumgeh}. Observe that this completes the proof of the lemma if $p=1.$ 

\noindent
For $p>1$ and $q=p/(p-1)$ we proceed with H\"older's inequality $\sum |x_iy_i|\le (\sum |x_i^p|)^{1/p}(\sum |x_i^q|)^{1/q}.$
\begin{equation}\label{eq:hoelder}
\begin{split}
 \sum_{ij\in E^\prime}|g_i^p-g_j^p|&=\sum_{ij\in E^\prime}|g_i-g_j|\frac{g_i^p-g_j^p}{g_i-g_j}\\
&\le \left(\sum_{ij\in E^\prime}|g_i-g_j|^p\right )^{1/p} \left(\sum_{ij\in E^\prime}\left(\frac{g_i^p-g_j^p}{g_i-g_j}\right)^q\right )^{1/q}
\end{split}
\end{equation}
Now observe that (assuming $g_i>g_j$)
\[\begin{split}
 \left(\frac{g_i^p-g_j^p}{g_i-g_j}\right)^{q}&=p^{q}\left(\frac{1}{g_i-g_j}\int_{g_j}^{g_i}t^{p-1}dt\right)^{p/(p-1)}\\
&\le p^{q}\left(\frac{1}{g_i-g_j}\int_{g_j}^{g_i}t^{p}dt\right)^{p/p}\le p^{q}\left(\frac{g_i^p+g_j^p}{2}\right)
\end{split}
\]
where the first ``$\le$'' is the power mean inequality applied to Riemann sums and the second ``$\le$'' is the convexity of $t^p:$ ($t^p$ is estimated by the secant line) (cf. \cite{Amghibech03}).

The second factor in the last term of \eqref{eq:hoelder} can thus be upper bounded by (with $d_i^\prime$ the degree of vertex $i$ in $G^\prime$)
\begin{equation*}\begin{split}
 \left(\sum_{ij\in E^\prime}\left(\frac{g_i^p-g_j^p}{g_i-g_j}\right)^q\right )^{1/q}\le p\left(\sum_{ij\in E^\prime}\frac{g_i^p+g_j^p}{2}\right )^{1/q}
=\frac{p}{2^{1/q}}\left(\sum_{i\in V^\prime}d_i^\prime g_i^p\right )^{1/q}\\ \le \frac{p}{2^{1/q}}\left(\Delta(G^\prime)\pn{g}^p \right )^{1/q}=p\left( \frac{\Delta(G^\prime)}{2}\right)^{(p-1)/p}\pn{g}^{p-1}
\end{split}
\end{equation*}
Substitute this last term into \eqref{eq:hoelder} combine with \eqref{eq:sumgeh} and regroup terms to obtain the assertion of the Lemma.
\end{proof}

\section{Basic properties of the signless $p$-Laplacian}

For $p>1$ a necessary condition for an $x\in \Sp$ to yield the minimum (resp. maximum) in \eqref{eq:q} is the existence of a Lagrange multiplier $\mu$ such that the \emph{eigenequations} 
\begin{equation}\label{eq:eig}
\frac{1}{p} \nabla_x Q_p(x)_i=\sum_{j\colon ij\in E}\sg{x_i+x_j} |x_i+x_j|^{p-1}=\mu \,\sg{x_i}|x_i|^{p-1},\,i\in V,
\end{equation}
are satisfied. If a pair $(x,\mu)\in (\R^V\setminus\{0\})\times \R$ satisfies \eqref{eq:eig} we call $x$ an eigenvector and $\mu$ an eigenvalue of $Q_p.$ Then \eqref{eq:eig} implies $p^{-1}x^\top\nabla_x Q_p(x)=Q_p(x)=\mu \pn{x}^p$ and hence eigenvalues are non-negative and $q_p(G)$ and $\lma$ are the smallest, resp. largest eigenvalues.

We first observe that $q_p(G)$ and $\lambda_p(G)$ do not increase when passing to subgraphs.
\begin{lemma}\label{lem:subgraph}
 Let $H$ be a subgraph of $G.$ Then $q_p(H)\le q_p(G)$ and $\lambda_p(H)\le\lambda_p(G).$ 
\end{lemma}
\begin{proof}
Since adding isolated vertices does not affect $\lma(H)$ we assume that $H=(V,F)$ with $F\subseteq E.$
Let $x\in \Sp$ be a normalised eigenvector for $\lambda_p(H),$ then
\[
\lambda_p(H)=\sum_{ij\in F}|x_i+x_j|^p\le \sum_{ij\in E}|x_i+x_j|^p\le \lambda_p(G). 
\]
The proof for $q_p$ is similar, starting with an eigenvector for $q_p(G).$
\end{proof}
The standard basis of $\R^V$ yields easy bounds for $q_p(G)$ and $\lma(G).$
\begin{lemma}\label{lem:ei}
We have for $p\ge 1$
\[
 q_p(G)\le \delta(G) \text{ and }\lma(G)\ge \Delta(G)
\]
($\delta(G),\Delta(G)$ the minimum/maximum degree). For $p>1$ equality holds in the former if and only if $G$ has an isolated vertex and equality in the latter if and only if $G$ has no edges. 

The standard basis vector $e_i$ ($i\in V$) of $\R^V$ is an eigenvector of $Q_p$ if and only if $d_i=0,$ that is $i$ is an isolated vertex.
\end{lemma}
\begin{proof}
By definition, $q_p(G)\le Q_p(e_i)=d_i\le\lma(G).$ If equality holds either of the two then $e_i$ is an eigenvector for the eigenvalue $d_i.$ If $d_i>0$ then $i$ has a neighbour $j$ and the $j$-th eigenequation reads $1=d_i\cdot 0,$ and so $d_i=0.$
\end{proof}
 
The next property is a ``mini-Perron-Frobenius''-Theorem for $\lma$.
\begin{theorem}
 Let $p>1,$ $G$ be connected and $x\in \Sp$ be an eigenvector for $\lma(G).$ Then $x$ has only stricly positive or only strictly negative components and is unique up to sign.
\end{theorem}
\begin{proof}
Denote by $C^+, C^-$ and $C^0$ the vertex sets on which $x_i$ is $>0$, $<0$ and $=0,$ respectively, and define $|x|=(|x_i|,i\in V)^\top.$  
Observe that $|x_i+x_j|\le ||x_i|+|x_j||$ with equality if and only if $x_i$ and $x_j$ have the same sign or at least one of them is $0.$
Thus $E(C^+,C^-)=\emptyset$ and $|x|$ is another eigenvector for $\lma$ because otherwise we had a contradiction $\lma=Q_p(x)< Q_p(|x|)\le\lma.$ Since $G$ is connected there is an edge $ab$ with $a\in C^+$ and $b\in C^0.$ The $b$-th eigenequation for the eigenpair $(|x|,\lma)$ then yields a contradiction
\[
0<||x_a|+|x_b||^{p-1} \le\sum_{j\colon bj\in E} ||x_b|+|x_j||^{p-1}=\lma \,\sg{x_b}x_b^{p-1}=0.
\]
So $C^0$ must be empty and either $C^+=V$ or $C^-=V.$

As for the uniqueness assume that $x,y\in\Sp$ are both stricly positive eigenvectors and define $z\in \Sp$ by 
\[
 z_i=\left(\frac{x_i^p+y_i^p}{2}\right)^{1/p},\,i\in V.
\]
The triangle inequality for the $p$-norm (Minkowski's inequality) yields 
\begin{equation}\label{eq:unique}
\begin{split}
 (x_i+x_j)^p+(y_i+y_j)^p&=\pn{\left( \begin{array}{c}
				  x_i \\y_i
                               \end{array}\right )
 +\left( \begin{array}{c}
				  x_j \\y_j
                               \end{array}\right )
 }^p\\
&\le \left( \pn{\left( \begin{array}{c}
				  x_i \\y_i
                               \end{array}\right )}
 +\pn{\left( \begin{array}{c}
				  x_j \\y_j
                               \end{array}\right )
 }  \right)^p\\
&=\left(\left(x_i^p+y_i^p\right)^{1/p}+\left(x_j^p+y_j^p\right)^{1/p}\right)^p\\
&=2(z_i+z_j)^p
\end{split}
\end{equation}
and thus $2\lma=Q_p(x)+Q_p(y)\le 2Q_p(z)\le 2\lma$ and equality must hold in \eqref{eq:unique} for every edge $ij\in E.$ Since the $p$-norm for $p>1$ is strictly convex $(x_i,y_i)$ must be a positive multiple of $(x_j,y_j)$ and again by connectedness of $G$ the rank of $(x,y)$ is one and hence $x=y.$  
 \end{proof}

\noindent
\textbf{Remark:} The theorem is false for $p=1.$ The unit ball is the convex hull of $\{\pm e_i,i\in V\}$ and so by the convexity of $Q_1$ we have that $\lambda_1=\max_{i\in V}Q_1(e_i)=\Delta,$ the maximum degree of $G.$ The solution is neither strictly positive nor is it unique, unless the maximum degree vertex is unique.

\smallskip
 
With the positive eigenvector at hand we can prove
\begin{lemma}\label{lem:degreebound}
 Let $p>1,$ $G=(V,E)$ be connected with maximum vertex degree $\Delta$ and minimum degree $\delta.$ Then
\[
 2^{p-1}\delta\le 2^{p-1}\frac{2|E|}{|V|}\le \lma(G)\le2^{p-1}\Delta
\]
with equality in either place if and only if $G$ is regular. In particular, the all ones vector is an eigenvector if and only if $G$ is regular.
\end{lemma}
\begin{proof}
The first ''$\le$`` is trivially true because the minimum degree is less than or equal to the average degree with equality only for regular graphs. For the second inequality observe that for the all ones vector $\bf{1}$ we have 
\[
 \lma(G)\ge\frac{Q_p(\bf{1})}{\pn{\bf{1}}^p}=\frac{2^p|E|}{n}
\]
with equality if and only if $\bf{1}$ is an eigenvector for $\lma(G).$ More generally, if $\bf{1}$ is an eigenvector for some eigenvalue $\mu$ then the eigenequations read $2^{p-1}d_i=\mu,i\in V.$ Thus $G$ is regular and $\mu=2^{p-1}\Delta\le \lma.$ 

 For the last inequality let $x$ be the positive eigenvector for $\lma$ and assume w.l.o.g. $x_1\ge x_2\ge\ldots \ge x_n>0.$ The first eigenequation reads
\[
 \lma x_1^{p-1}=\sum_{j\colon 1j\in E} |x_1+x_j|^{p-1}\le 2^{p-1}x_1^{p-1} d_1\le 2^{p-1}\Delta x_1^{p-1}.
\]
If equality holds then $d_1=\Delta$ and $x_a=x_1$ for every neighbour $a$ of $1.$ Then the eigenequation for $a$ yields $2^{p-1}\Delta x_a^{p-1}\le 2^{p-1}x_a^{p-1} d_a$ and so $d_a=\Delta$ and $x_b=x_a$ for any neighbour $b$ of $a.$ Continuing in this ``breadth first search'' fashion shows that $G$ is regular and $x_1=x_2=\ldots =x_n.$ 
\end{proof}
\noindent
\textbf{Remark:} For $p=1$ the upper bound is always attained and equality holds in the lower bounds if and only if $G$ is regular. 

A better upper bound is the following.
\begin{lemma}
 Let $p>1,$ $q=p/(p-1)$ and $G$ be connected. Then
\[
 \lma(G)\le 2^{p-1}\max_{ij\in E} \left(\frac{d_i^q+d_j^q}{2}   \right)^{1/q}
\] with equality if and only if $G$ is regular.
\end{lemma}
\begin{proof}
 Let $x$ be a positive eigenvector for $\lma$ and and choose an edge $ij$ such that $x_i+x_j$ is maximal. The $i$-th eigenequation yields the estimate 
\begin{equation}\label{eq:lx}
 \lma x_i^{p-1}=\sum_{k\colon ik\in E}(x_i+x_k)^{p-1}\le d_i(x_i+x_j)^{p-1}.
\end{equation}
Take the $q$-th power on both sides and use the convexity of $t\mapsto t^p$ to obtain
\begin{equation}\label{eq:lqxp}
 \lma^q x_i^{p}\le d_i^q(x_i+x_j)^{p}\le 2^{p-1}d_i^q(x_i^p+x_j^p).
\end{equation}
The same computation for $i$ replaced by $j$ yields $\lma^q x_j^{p}\le 2^{p-1}d_j^q(x_i^p+x_j^p)$ and adding up the two yields
\[
 \lma^q (x_i^p+x_j^{p})\le 2^{p}\frac{d_i^q+d_j^q}{2}(x_i^p+x_j^p)
\]
and hence the bound. If the bound and thus \eqref{eq:lqxp} hold with equality then $x_i=x_j$ by the strict convexity of $t\mapsto t^p$ and therefore $d_i=d_j.$ From \eqref{eq:lx} we get $x_k=x_j=x_i$ for every $k$ with $ik\in E$ and we can argue similarly as in the proof of Lemma~\ref{lem:degreebound} that $G$ is regular.
\end{proof}

Before we generalise some more known inequalities we make some remarks on complete graphs and odd cycles.

\noindent
\textbf{Example:} Complete graphs $K_n.$ By Lemma \ref{lem:degreebound} we see that $\lma(K_n)=2^{p-1}(n-1).$ Furthermore, $e_i-e_j,i\ne j,$ is an eigenvector of $Q_p$ for $\mu=n-2,$ regardless of the value of $p$ and for $p=2$ this yields the whole spectrum. One can show directly that $\psi(K_n)=n-2$ (with $S=\{1\}$ and $T=\{2\}$ in \eqref{eq:psi}). For $p\neq 2$ the smallest eigenvalue is less than $n-2:$
Consider the vector $x=(n-1,-1,\ldots,-1)^\top.$ For $p\neq 0$ this is not an eigenvector and therefore we have a strict inequality
\[
q_p(K_n)< \frac{Q_p(x)}{\pn{x}^p}=(n-2)\frac{(n-2)^{p-1}+2^{p-1}}{(n-1)^{p-1}+1}
\]
which for $p>2$ is  $<n-2$ and for $p>2$ is $>n-2.$ The vector $x=(1,1,-1,-1,0,\ldots,0)^\top$ is an eigenvector for $\mu=n-4+2^{p-1}$ which is strictly less than $n-2$ if $p<2.$ 

\noindent
\textbf{Example:} Odd cycle $C_n=(V=\{1,\ldots,n\}, E=\{12,23,\ldots,(n-1) n,n1\}).$ The largest eigenvalue is $\lma(C_n)=2^p.$ For the smallest eigenvalue Lemma \ref{lem:upperbound} with $S=\{1,3,\ldots,n-2\}$ and $T=\{2,4,\ldots,n-1\}$ yields $q_p(C_n)\le 2/(n-1)$ which is $<1$ if $n\ge 5.$

\smallskip

The following eigenvalue inequalities are simple generalisations of known ones in the case $p=2.$ The proofs carry over almost verbatim and we refer to the original sources. Wilf's bound is originally for the adjacency matrix $A$ and true for $Q$ by the relation $\lambda_2(G)\ge 2\mu$ where $\mu$ denotes the largest eigenvalue of $A$. The proof is not long so we decided to include it here.
\begin{prop}
 Let $G=(V,E)$ be connected, $n=|V|,$ $m=|E|$ and $\chi=\chi(G)$ be the chromatic number.
\begin{enumerate}
 \item $2^{p-1}(\chi-1)\le \lma(G)$ with equality if and only if $G$ is complete or an odd cycle (Wilf's bound \cite{Wilf67}).
\item $q_p(G)\le \frac{2m}{n}\cdot \frac{\chi-2}{\chi-1} \cdot \frac{(\chi-2)^{p-1}+2^{p-1}}{(\chi-1)^{p-1}+1}.$ If $p=2$ and $G$ is complete then equality holds \cite[Theorem 2.11]{LOAN11}. 
\item $\lma(G)-q_p(G)\ge 2^{p-1}(\chi-1)-(\chi-2)\frac{(\chi-2)^{p-1}+2^{p-1}}{(\chi-1)^{p-1}+1}.$ If equality holds then $G$ is a complete graph and $p= 2$ \cite[Corollary 2.12]{LOAN11}.
\item Denote by $\nu(G)$ the vertex bipartiteness of $G,$ i.e. the minimum cardinality of a set $S\subseteq V$ such that $G[V\setminus S]$ is bipartite. Then $q_p(G)\le \nu(G).$ \cite[Theorem 2.1]{FallatFan12}
\end{enumerate}
\end{prop}
\begin{proof}
 1. Remove vertices from $G$ to obtain a $\chi-$critical subgraph $H=(V(H),E(H))$ (i.e. $\chi(H)=\chi$ and $H-i$ is $\chi-1$-colourable for every $i\in V(H)$). Then $H$ is connected and the minimum degree $\delta(H)$ is at least $\chi-1.$ So by Lemmas \ref{lem:degreebound} and \ref{lem:subgraph} we have the desired bound:
\[
 2^{p-1}(\chi-1)\le 2^{p-1}\delta(H)\le \lma(H)\le \lma(G).
\]
If equality holds in item 1 then $\lma(H)=\lma(G)$ and, again by Lemma \ref{lem:degreebound}, $H$ is $(\chi-1)$-regular and the all ones vector ${\bf 1}\in\R^{V(H)}$ is an eigenvector for $\lma(H).$ We show that $V(H)=V.$ Define $x\in\R^V$ by $x_i=1$ if $i\in V(H)$ and $x_i=0$ otherwise. Then, in $G$ we get
\[
\lma(G)\ge\frac{Q_p(x)}{\pn{x}^p}=\frac{2^p|E(H)|+\cut_G(V(H))}{|V(H)|}=\lma(H)+\frac{\cut_G(V(H))}{|V(H)|}\ge \lma(G)
\]
and so $\cut_G(V(H))=0.$ Connectedness of $G$ implies $V=V(H).$ So $G$ is a $\chi$-chromatic, $\chi-1$-regular graph and therefore by the theorem of Brooks \cite{Brooks41} a complete graph or an odd cycle. 

\noindent
2. As in \cite[Theorem 2.11]{LOAN11}. The second assertion is shown in the above example. 

\noindent
3. Here we slightly deviate from \cite{LOAN11} but get the same result for $p=2.$ With 2 we have
\[\begin{split}
 \lma-q_p&\stackrel{\textnormal{item 2}}{\ge} \lma -\frac{2m}{n}\cdot \frac{\chi-2}{\chi-1} \cdot \frac{(\chi-2)^{p-1}+2^{p-1}}{(\chi-1)^{p-1}+1}\\
&\stackrel{\textnormal{Lem. }\ref{lem:degreebound}}{\ge} \lma-\frac{\lma}{2^{p-1}}\cdot \frac{\chi-2}{\chi-1} \cdot \frac{(\chi-2)^{p-1}+2^{p-1}}{(\chi-1)^{p-1}+1}\\
&\stackrel{\textnormal{item 1}}{\ge}2^{p-1}(\chi-1)-(\chi-2)\frac{(\chi-2)^{p-1}+2^{p-1}}{(\chi-1)^{p-1}+1}.
\end{split}
\]
If equality holds in the last step then $G$ is the complete graph $K_n$ or an odd cycle $C_n$ by 1. If $G=C_n,$ $n\ge 5,$ then the first inequality is strict because the bound in item 2 for an odd cycle reads $q_p(C_n)\le 1$ and actually $q_p(C_n)< 1$ by the above remarks on $C_n.$ Therefore $G=K_n$ and item 2 reads $q_p(K_n)\le (n-2)\frac{(n-2)^{p-1}+2^{p-1}}{(n-1)^{p-1}+1}$ and the above remarks showed that the inequality is strict for $p\neq 2.$ Hence equality in item 3 for $p\neq 2$ is impossible.    

\noindent
4. As in \cite[Theorem 2.1]{FallatFan12}.
\end{proof}

\noindent
\textbf{Remark:} In \cite[Theorem 2.11]{LOAN11} (item 2 for $p=2$) it says ''If $G$ is a regular $\chi$-partite graph then equality holds.`` This is in general not true as odd cycles ($n\ge 5$) show. Their proof of Corollary 2.12 is unaffected by this as it uses their Theorem 2.10.

\smallskip

\noindent
\textbf{Remark:} In the limit $p\to 1$ item 1 becomes Brooks' Theorem, items 2--4 become $\psi\le \frac{2m}{n}\cdot \frac{\chi-2}{\chi-1},$ $\psi\le \Delta -1$ and $\psi\le \nu,$ respectively.  The latter is easily seen combinatorially: For $S\subseteq V$ with $G-S$ bipartite and $|S|=\nu$ we have by the definition of $\psi:$ $\psi\le \frac{\cut(V\setminus S)}{n-\nu}=\frac{\cut(V\setminus S)}{(n-\nu)\nu}\nu\le \nu.$


\section{Concluding remarks}

In \cite{Chang16} spectral properties of the 1-Laplacian are explored. The eigenequations are then a nonlinear system involving set valued functions (gradients are replaced by subdifferentials) and it is found that the smallest non-zero eigenvalue is indeed $i(G).$ A similar study might be possible for the signless 1-Laplacian $Q_1$ as well.

As for practical matters it would be interesting to know how well the method proposed by Corollary \ref{coro:main} can be put to use for the questions raised in \cite{KirklandPaul11}. It might be possible to adapt B\"uhler's and Hein's $p$-spectral clustering method to the computation of $\psi.$

\smallskip

\noindent
\textbf{Acknowledgments}
This work was partially supported by CAPES Grant PROBRAL 408/13 - Brazil
and DAAD PROBRAL Grant 56267227 - Germany.

\end{document}